\documentclass[12pt]{article}
\usepackage{bbm}
\usepackage{amsmath}
\usepackage[hidelinks]{hyperref}
\allowdisplaybreaks

\renewcommand{\emph}[1]{\textit{#1}}
\usepackage{enumerate,amsmath,amsthm,latexsym,amssymb}
\usepackage{color}\usepackage{graphicx}

\definecolor{brown}{cmyk}{0, 0.72, 1, 0.45}
\definecolor{grey}{gray}{0.5}

\newcommand{\old}[1]{}

\newcounter{rot}

\newcommand{\ignore}[1]{}

\newcommand{\set}[1]{\left\{#1\right\}}

\def\cP{\mathcal{P}}

\def\ii_(#1,#2){i_{#1}^{#2}}

\def\a{\alpha}
\def\b{\beta}

\def\e{\varepsilon}

\def\G{\Gamma}

\def\s{\sigma}
\def\t{\tau}

\def\1{{\bf 1}}
\def\0{{\bf 0}}

\newcommand{\brac}[1]{\left( #1 \right)}

\renewcommand{\Pr}{\operatorname{\bf Pr}}
\newcommand\bfrac[2]{\left(\frac{#1}{#2}\right)}

\newcommand{\nospace}[1]{}

\def\path{\operatorname{PATH}}

\newcommand{\mult}[2]{\begin{multline}\label{#1}#2\end{multline}}
\newcommand{\mults}[1]{\begin{multline*}#1\end{multline*}}
\newcommand{\beq}[2]{\begin{equation}\label{#1}#2\end{equation}}

\newtheorem{theorem}{Theorem}[section]

\newtheorem{lemma}[theorem]{Lemma}

\newtheorem{remthm}[theorem]{Remark}

\newtheorem{definition}[theorem]{Definition}
\newtheorem{notation}[theorem]{Notation}

\newenvironment{remark}{\begin{remthm}\it }{\end{remthm}}%
\newcounter{thmtemp}

\usepackage[ruled, linesnumbered]{algorithm2e}
\def\cP{{\mathcal P}}

\begin{document}
\author{Michael Anastos and Alan Frieze\thanks{Research supported in part by NSF grant DMS1661063}\\Department of Mathematical Sciences\\Carnegie Mellon University\\Pittsburgh PA 15213} 

\title{On the connectivity threshold for colorings of random graphs and hypergraphs}
\maketitle
\begin{abstract}
Let $\Omega_q=\Omega_q(H)$ denote the set of proper $[q]$-colorings of the hypergraph $H$. Let $\G_q$ be the graph with vertex set $\Omega_q$ and an edge $\set{\s,\t}$ where $\s,\t$ are colorings iff $h(\s,\t)=1$. Here $h(\s,\t)$ is the Hamming distance $|\set{v\in V(H):\s(v)\neq\t(v)}|$. We show that if $H=H_{n,m;k},\,k\geq 2$, the random $k$-uniform hypergraph with $V=[n]$ and $m=dn/k$ then w.h.p. $\G_q$ is connected if $d$ is sufficiently large and $q\gtrsim (d/\log d)^{1/(k-1)}$. Furthermore, with a few more colors, we find that the diameter of $\G_q$ is $O(n)$ w.h.p.
\end{abstract}
\section{Introduction}\label{intro}
In this paper, we will discuss a structural property of the set $\Omega_q$ of proper $[q]$-colorings of the random hypergraph $H=H_{n,m;k}$, where $m=dn/k$ for some large constant $d$. Here $H$ has vertex set $V=V(H)=[n]$ and an edge set $E=E(H)$ consisting of $m$ randomly chosen $k$-sets from $\binom{[n]}{k}$. Note that in the graph case where $k=2$ we have $H_{n,m;2}=G_{n,m}$. A proper $[q]$-coloring is a map $\s:[n]\to [q]$ such that $|\s^{-1}(e)|\geq 2$ for all $e\in E$ i.e. no edge is mono-chromatic. Then let us define $\G_q=\G_q(H)$ to be the graph with vertex set $\Omega_q$ and an edge $\set{\s,\t}$ iff $h(\s,\t)=1$ where $h(\s,\t)$ is the Hamming distance $|\set{v\in [n]:\s(v)\neq\t(v)}|$. In the Statistical Physics literature the definition of $\G_q$ may be that colorings $\s,\t$ are connected by an edge in $\G_q$ whenever $h(\s,\t)=o(n)$. Our theorem holds a fortiori if this is the case.

{\bf Notation:} $f(d)\gtrsim g(d)$ if $f(d)\geq (1+\e_d)g(d)$ for $d$ large and where $\e_d>0$ and $\lim_{d\to\infty}\e_d=0$.

Then let 
\beq{defab}{
\alpha=\bfrac{(k-1)d}{\log d-5(k-1)\log \log d}^{\frac{1}{k-1}},\quad\beta=3\log^{3k}d.
}
We prove the following.  
\begin{theorem}\label{th1}
Suppose that $k\geq 2$ and $p=\frac{d}{\binom{n-1}{k-1}}$ and $m=\binom{n}{k}p$ and that  $d=O(1)$ is sufficiently large. Then 
\begin{enumerate}[(i)]
\item If $q\geq \a+\b+1$ then w.h.p. $\G_q$ is connected.
\item If $q\geq \a+2\b+1$ then the diameter of $\G_q$ is $O(n)$ w.h.p.
\end{enumerate}
\end{theorem}
Note that $\G_q$ connected implies that ``The Glauber Dynamics on $\Omega_q$ is ergodic''. At the moment we only know that Glauber Dynamics is rapidly mixing for $q\geq (1.76\ldots) d$, see Efthymiou, Hayes, \v{S}tefankovi\v{c} and Vigoda \cite{EHSV}. So, it would seem that the connectivity of $\G_q$ is not likely to be a barrier to randomly sampling colorings of sparse random graphs.

We note that the lower bound for $q$ is close to where the greedy coloring algorithm succeeds w.h.p.

We should note that in the case $k=2$ that Molloy \cite{M} has shown that w.h.p. there is no giant component in $\G_q$ if $q\lesssim \frac{d}{\log d}$. It is somewhat surprising therefore that w.h.p. $\G_q$ jumps very quickly from having no giant to being connected. One might have expected that $q\gtrsim \frac{d}{\log d}$ would simply imply the existence of a giant component.

Prior to this paper, it was shown in \cite{DFFV} that w.h.p. $\G_q,q\geq d+2$ is connected. The diameter of the {\em reconfiguration graph} $\G_q(G)$ for graphs $G$ has been studied in the graph theory litrature, see Bousquet and Perarnau \cite{BP} and Feghali \cite{Feg}. They show that if the maximum sub-graph density of a graph is at most $d-\e$ and $q\geq d+1$ then  $\G_q(G)$ has polynomial diameter. Using Theorem 1 of \cite{BP} we can show a linear bound on the diameter with a small increase in the number of colors.

Theorem \ref{th1} falls into the area of ``Structural Properties of Solutions to Random Constraint Satisfaction Problems''. This is a growing area with connections to Computer Science and Theoretical Physics.  In particular, much of the research on the graph $\G_q$ has been focussed on the structure near the {\em colorability threshold}, e.g. Bapst, Coja-Oghlan, Hetterich, Rassman and Vilenchik \cite{BCHRV}, or the {\em clustering threshold}, e.g. Achlioptas, Coja-Oghlan and Ricci-Tersenghi \cite{ACR}, Molloy \cite{M}.  Other papers heuristically identify a sequence of phase transitions in the structure of $H_q$, e.g., Krz\c{a}kala, Montanari, Ricci-Tersenghi, Semerijan and Zdeborov\'a \cite{KMRSZ}, Zdeborov\'a and Krz\c{a}kala \cite{ZK}.  The existence of these transitions has been shown rigorously for some other CSPs. One of the most spectacular examples is due to Ding, Sly and Sun \cite{sly} who rigorously showed the existence of a sharp satisfiability threshold for random $k$-SAT.

Section \ref{alg} describes a property $(\a,\b)$-colorability such that if $H$ has this property then $q\geq \a+\b+1$ implies that $\G_q$ is connected. Section \ref{hcase} proves that $H_{n,m;k},k\geq 2$, is $(\a,\b)$-colorable for $\a,\b$ defined in \eqref{defab}.

The paper uses some of the ideas from \cite{AFP} which showed there is a giant component in $\G_q(G_{n,m}),\,m=dn/2$ w.h.p. when $q\geq cd/\log d$ for $c>3/2$.
\section{Outline argument}
We show that with the values $\a\approx ((k-1)d/\log d)^{1/(k-1)}\gg \b$ given in \eqref{defab} then w.h.p. $H=H_{n,m;k}$ has the property that {\bf any} greedy coloring of $H$ will need at most $\a$ maximal independent sets before being left with a graph without a $\b$-core. (See Lemma \ref{indeph}.)  We call the colorings found in this way, {\em good greedy colorings} and we refer to this property as $(\a,\b)$-colorability. It follows from this, basically using the argument from \cite{AFP}, that if $\s\in \Omega_q$ and $q\geq \a+\b+1$ then there is a {\em good} path in $\G_q$ to some good greedy coloring $\s_1$. 

Suppose now that $\s_1,\t_1$ are good greedy colorings. If $q\geq \a+\b+1$ then there is a color $c$ that is not used by $\s_1$. From $\s_1$ we move to $\s_2$ by re-coloring vertices colored 1 in $\s_1$ by $c$. Then we move from $\s_2$ to $\s_3$ by coloring with color 1, all vertices that have color 1 in $\t_1$. At this point, $\s_3$ and $\t_1$ agree on color 1. $\s_3$ may use $\a+\b+1$ colors and so we move by a good path from $\s_3$ to a coloring $\s_4$ that uses at most $\a+\b+1$ colors and does not change the color of any vertex currently with color 1. Here we use the fact that $H_{n,m;k}$ is $(\a,\b)$-colorable. After this, it is induction that completes the proof.
\section{$(\a,\b)$-colorability}\label{alg}
The degree of a vertex $v\in V$ in a hypergraph $H=(V,E)$ is the number of edges $e\in E$ such that $v\in e$. (For completeness, we will state several things in this short paper that one might think can be taken for granted.)

Let $H=(V,E)$. A $\beta$-core of $H$ is a maximal subgraph of $H$ in which every vertex has degree at least $\beta$. For every $U\subset V$, if the subgraph of $H$ induced by $U$ does not have a $\beta$-core then there is an ordering $\set{u_1, u_2,...,u_{|U|}}$ of the vertices in $U$ such that every vertex in $U$ has at most $\beta-1$ neighbors that precede it in that ordering.

If a hypergraph $H$ that does not have a $\beta$-core then we can color it with at most $\beta$ colors. Let $v_1,v_2,...,v_n$ be an ordering on $V$ where 
\mult{order}{
\text{for every $i$, there are at most $\beta-1$ edges that contain $v_i$}\\
\text{ and are contained in $\set{v_1,v_2,\ldots,v_i}$.}
}
Such an ordering must exist when there is no $\b$-core. We color the vertices in the order $v_1,v_2,\ldots,v_n$ and assign to $v_i$ a color that is not {\em blocked} by the $\beta-1$ neighbors that precede it. A color $c$ is blocked for vertex $v$ by vertices $w_1,w_2,\ldots,w_{k-1}$ if $e=\set{v,w_1,\ldots,w_{k-1}}\in E(H)$ and $w_1,w_2,\ldots,w_{k-1}$ have already been given color $c$.

Next let $V_1,V_2,\ldots,V_\a$ be a sequence of independent sets of $H$ such that for each $j\geq 1$, $V_{j}$ is maximal in the sub-hypergraph $H_{j}$ induced by $V\setminus \underset{1\leq i<j}{\bigcup}V_i$. We say that such a sequence is a {\em maximally independent sequence of length $\a$}. Note that we allow $V_j=\emptyset$ here, in order to make our sequences of length exactly $\a$.
\begin{definition}
We say that a hypergraph $H$ is $(\alpha,\beta)$-colorable if there {\bf does not exist} a maximally independent sequence of length $\a$ such that $V\setminus \underset{i\leq \a}{\bigcup}V_i$ has a $\b$-core.
\end{definition}
The main result of this section is the following.
\begin{theorem}\label{semigreedy}
Let $H$ be $(\alpha,\beta)$-colorable and let $q\geq \alpha+\beta +1$. Then $\Gamma_q(H)$ is connected.
\end{theorem}
Later, in Section \ref{hcase} we will show that $H_{n,m;k},k\geq 2$ is $(\alpha,\beta)$-colorable, for a suitable values of $m,\a,\b$.
\begin{lemma}\label{reduction}
Let $H=(V,E)$ be an $(\alpha,\beta)$-colorable hypergraph and $V_1\subseteq V$ be a maximal independent set of $V$. Set $V'=V\setminus V_1$ and let $H'$ be the subgraph of $H$ induced by $V'$. Then $H'$ is $(\alpha-1, \beta)$-colorable.
\end{lemma}
\begin{proof}
Assume that $H'$ in not $(\alpha-1,\beta)$-colorable. Then there exists a partition of $V'$ into $V_1',...,V_{\alpha-1}'$ such that for $j\in [\alpha-1]$, $V_j'$ is a maximal independent set of $V'\setminus \underset{\ell<j}{\bigcup}V_\ell'$ and $W'=V'\setminus \underset{\ell\leq \alpha-1}{\bigcup}V_\ell'$ has a $\beta$-core. For $j\in[\alpha-1]$ set $V_{j+1}=V_j'$. Furthermore set $W=V\setminus (\underset{\ell\leq \alpha}{\bigcup}V_\ell)=V'\setminus (\underset{\ell\leq \alpha-1}{\bigcup}V_\ell')=W'$. Then
$V_1,...,V_\alpha$ is a maximal independent sequence of length $\a$ and $W$ has a $\beta$-core which contradicts the fact that $H$ is $(\alpha, \beta)$-colorable. 
\end{proof}
\begin{lemma}\label{core}
Let $H$ be a hypergraph, $\alpha,\beta\geq 0$ and $q\geq \alpha+\beta+1$. Let $W\subseteq V$ be such that the subgraph of $H$ induced by $W$ has no $\beta$-core. Furthermore let $\chi$ and $\tau$ be two colorings of $H$ such that 
\begin{enumerate}[(i)]
\item They agree on $V\setminus W$. 
\item They use only $\alpha$ colors on the vertices in $V\setminus W$.
\item $\tau$ uses at most $\beta$ colors on $W$ that are distinct from the ones it uses on $V\setminus W$.
\end{enumerate}
Then there exists a path from $\chi$ to $\tau$ in $\Gamma_q(H)$. 
\end{lemma}
\begin{proof}
Without loss of generality we may assume that $\chi$ and $\tau$ use $[\alpha]$ to color $V\setminus W$. The proof that follows is an adaptation to hypergraphs of the proof in \cite{AFP} that $\G_q(G)$ is connected when a graph $G$ has no $q$-core. Because $W$ has no $\beta$-core there exists an ordering of its vertices, $v_1,v_2,...,v_r$, such that for $i\in [r]$, $v_i$ has at most $\beta-1$ neighbors in $v_1,v_2,\ldots,v_{i-1}$. For $0\leq i\leq r$ let $\tau_i$ be the coloring that agrees with $\tau$ on $\set{v_1,...,v_i}$ and with $\chi$ on $W\setminus \set{v_1,...,v_i}$. On $V\setminus W$ it agrees with both. Thus $\t_0=\chi$ and $\t_r=\t$.

We proceed by induction on $i$ to show that there is a sequence of colorings $\Sigma_i$ from $\chi$ to $\t_i$ such that (i) going from one coloring to the next in $\Sigma_i$ only re-colors one vertex and (ii) all colorings in the sequence $\Sigma_i$ are proper for the hypergraph induced by $V\setminus \set{v_{i+1},...,v_r}$. We {\bf do not} claim that the colorings in $\Sigma_i,i<r$ are proper for $H$. On the other hand, taking $i=r$ we get a sequence of $H$-proper colorings that starts with $\chi$, ends with $\tau$, such that the consecutive pairs of proper colorings differ on a single vertex. Clearly, such a sequence corresponds to a path from $\chi$ to $\tau$ in $\Gamma_q(H)$. 

The case $i=1$ is trivial as we have assumed that $\s,\t$ agree on $V\setminus W$ and so we can give $v_1$ the color $\t(v_1)$. Assume that the assertion is true for $i=k\geq 1$ and let  $\chi=\psi_0,\psi_1,\ldots,\psi_s =\t_k$ be a sequence of colorings promised by the inductive ssertion. Let $(w_j,c_j)$ denote the $(vertex,color)$ change defining the change from $\psi_{j-1}$ to $\psi_j$. We construct a sequence of colorings of length at most $2s+1$ that yields the assertion for $i=k+1$. For $j=1,2,\ldots,s$, we will re-color $w_j$ to color $c_j$, unless there exists a set $X$  such that $X\cup \set{w_j}\in E$ and $\psi_{j-1}(x)=c_j,x\in X\subseteq \set{v_1,v_2,\ldots,v_{k+1}}$. The fact that $\psi_j$ is a proper coloring of $V\setminus \set{v_{k+1},...,v_r}$ implies that $v_{k+1}\in X$. Because $v_{k+1}$ has at most $\beta-1$ neighbors in $\set{v_1,...,v_k}$ and $\tau$ only uses colors in $[\alpha]$ to color $V\setminus W$, there exists a color $c'\neq c_j$ for $v_{k+1}$ in $[\alpha+\beta+1]\setminus [\alpha]$ which is not blocked by a subset of $\set{v_1,v_2,\ldots,v_k}$ and is diferent from its current color. We first re-color $v_{k+1}$ to $c'$ and then we re-color $w_j$ to $c_j$, completing the inductive step. At the very end, i.e. at step $2s+1$ we give $v_{k+1}$ its color in $\t$.
\end{proof}
\begin{definition}
A coloring with color sets $V_1,V_2,\ldots,V_{ \a+\b}$ is said to be a {\em good greedy coloring} if  (i) $V_1,V_2,\ldots,V_\a$ is a maximally independent sequence of length $\a$ and (ii) $V\setminus \underset{\ell\leq \a}{\bigcup}V_\ell$ has no $\b$-core.
\end{definition}
We prove Theorem \ref{semigreedy} in two steps. In Lemma \ref{inter}, we show that if $q\geq \a+\b+1$ and $H$ is $(\a,\b)-colorable$ then we can reach a good greedy coloring in $\Gamma_q(H)$ starting from any coloring. Then in Lemma \ref{final}, we show that if $q\geq \a+\b+1$ then any good greedy coloring $\t$ can be reached in $\G_q(H)$ from any other good greedy coloring $\s$.
\begin{lemma}\label{inter}
Let $H$ be an $(\alpha,\beta)$-colorable hypergraph, $q\geq \alpha+\beta+1$ and $\chi$ be a $[q]$-coloring of $H$. Then there exists a good greedy coloring $\tau$  of $H$ such that there exists a path in $\Gamma_q(H)$  from $\chi$ to $\tau$.   
\end{lemma}
\begin{proof}
We generate the coloring $\tau$ as follows. Let $C_1,C_2,\ldots,C_q$ be the color classes of $\chi$. Then let $V_1\supseteq C_1$ be a maximal independent set containing $C_1$. In general, having defined $V_1,V_2,\ldots,V_{\ell-1}$ we let $V_{<\ell}= \underset{1\leq i<\ell}{\bigcup} V_i$ and then we let $V_\ell$ be a maximal independent set in $V\setminus V_{<\ell}$ that contains $C_\ell\setminus V_{<\ell}$. Thus $V_1,V_2,\ldots,V_\a$ is a maximal independent sequence of length $\a$. We now describe how we transform the coloring $\chi$ vertex by vertex into a coloring $\chi'$ in which vertices in $V_i$ get color $i$ for $1\leq i\leq \a$. We first re-color the vertices in $V_1\setminus C_1$ by giving them color 1, one vertex at a time. The coloring stays proper, as $V_1$ is an independent set. In general, having re-colored  $V_1,V_2,\ldots,V_{\ell-1}$ we re-color the vertices in $V_\ell\setminus C_\ell$ with color $\ell$. Again, the coloring stays proper, as $V_\ell$ is an independent set, containing all vertices in $C_\ell$ that have not been re-colored. We observe that each re-coloring of a vertex $v$ done while turning $\chi$ into $\chi'$ can be interpreted as moving from a coloring in $\Gamma_q(H)$ to a neighboring coloring.

Let $W= V\setminus \underset{1\leq i < \ell}{\bigcup} V_i$. Because $H$ is $(\alpha,\beta)$-colorable, we find that $W$ has no $\beta$-core. Because $W$ has no $\beta$-core there exists a proper coloring $\t'$ of the subgraph of $H$ induced by $W$ that uses only colors in $[\alpha+\beta+1]\setminus [\alpha]$. Set $\tau$ to be the coloring that agrees with $\chi'$ on $V\setminus W$ and with $\tau'$ on $W$. 

Lemma \ref{core} implies that there is a path from $\chi'$ to $\tau$. Hence there is a path from $\chi$ to $\tau$.
\end{proof}
\begin{remark}\label{rem1}
In the proof of Lemma \ref{inter} we see that each vertex is re-colored at most twice before we apply Lemma \ref{core}. Thus this part of the proof yields at most $\a$ distinct sub-paths of length $O(n)$.
\end{remark}

\begin{lemma}\label{final}
Let $H$ be an $(\alpha,\beta)$-colorable hypergraph, $q\geq \alpha+\beta+1$ and let $\chi,\t$ be two good greedy colorings. Then there exists a path from $\chi$ to $\tau$ in $\Gamma_q(H)$.
\end{lemma}
\begin{proof}
We proceed by induction on $\alpha$. For $\alpha=0$, $H$ is $(0,\beta)$ colorable and so it does not have a $\beta$-core. Thus the base case follows directly from Lemma \ref{core} by taking $W=V$.

Assume that the statement of the Lemma is true for $\alpha=k-1$ and let $\alpha=k$. There exists a maximal independent sequence $V_1,V_2,\ldots,V_k$ of length $k$ such that if $V'=V\setminus \underset{1\leq i\leq k}{\bigcup}V_i$ then (i) for $i \in [k]$, $\tau$ assigns the color $i$ to $v\in V_i$ and (ii) $\tau$ assigns only colors in $[k+\beta]\setminus [k]$ to vertices in $V'$.

Let $c$ be a color not assigned by $\chi$. There is one as $q\geq k+\b+1$. Starting from $\chi$ we recolor all vertices that are colored 1 by color $c$ to create a coloring $\bar{\chi}$. Then we continue from $\bar{\chi}$ by recoloring all the vertices in $V_1$ by color 1 and we let $\chi'$ be the resulting coloring. Clearly there is a path $P_1$ from $\chi$ to $\chi'$ in $\Gamma_q(H)$.

We now set  $H_1=H\setminus V_1$, and set $\chi_1',\tau_1$ to be the restrictions of $\chi',\tau$ on $H_1$. Observe that since $V_1$ is a maximal independent set, Lemma \ref{reduction} implies that $H_1$ is $(k-1,\beta)$ colorable and in addition that $\tau_1$ is a good greedy coloring of $H_1$. Lemma \ref{inter} implies that in $\Gamma_{q-1}(H_1)$ there is a path $P_2$ from  $\chi_1'$ to some good greedy coloring $\chi_1$ that uses only $k-1+\beta$ colors from $[q]\setminus \{1\}$. The induction hypothesis implies that in $\Gamma_{q-1}(H_1)$ that there is a path $P_3$ from $\chi_1$ to $\tau_1$.

Color 1 is not used in $\chi_1', \tau_1$ or in any of colorings found in the path $P_2,P_3$. Thus the path $P_2,P_3$ corresponds to a path $P_4$ in $\Gamma_q(H)$ from $\chi'$ to $\tau$.
Consequently the colorings $\chi,\tau$ are connected in $\Gamma_q(H)$ by the path $P_1+P_4$.
\end{proof}
{\emph{Proof of Theorem \ref{semigreedy}:}
Let $H$ be $(\alpha,\beta)$ colorable, $q\geq \alpha+\beta +1$, and let $\chi_1,\chi_2$ be two colorings of $H$. Lemma \ref{inter} implies that in $\Gamma_{q}(H)$, there ia path $P_i$ from $\chi_i$ to a good greedy coloring $\t_i$ for $i=1,2$. Lemma \ref{final} implies that there is a path in $\G_q(H)$ from $\t_1$ to $\t_2$.
\qed
\section{Random Hypergraphs}\label{hcase}

Theorem \ref{th1} follows from
\begin{lemma}\label{randomh}
Let  $k\geq 2$ and suppose that $q\geq \a+\b+1$ and that $d$ is sufficiently large. If $p=\frac{d}{\binom{n-1}{k-1}}$ and $m=\binom{n}{k}p$ then w.h.p. $\Gamma_q(H_{n,m;k})$ is connected.
\end{lemma}
In the following we will assume for simplicity of notation that $d=O(1)$, so that $O(f(d)/n)=O(1/n)$. We do not know if there is an upper bound needed for the growth rate of $d$, but we doubt it. 

To prove Lemma \ref{randomh} we use Lemmas \ref{auxh}, \ref{indeph}, \ref{densityh} (below) in order to deduce that w.h.p. $G_{n,dn/2}$ is $(\alpha,\beta)$ colorable. Then we apply Theorem \ref{semigreedy}. (Lemmas \ref{auxh} and \ref{densityh} are hardly new or best possible, but we prove them here for completeness.)

We will do our calculations on the random graph $H_{n,p;k},p=d/\binom{n-1}{k-1}$ and use the fact for any hypergraph property $\cP$, we have
\beq{cP}{
\Pr(H_{n,m;k}\in\cP)\leq O(m^{1/2})\Pr(H_{n,p;k}\in\cP).
}
\begin{lemma}\label{auxh}
Let $p=\frac{d}{\binom{n-1}{k-1}}$ and $k\geq 2$ and $d$ sufficiently large. Then, w.h.p. $H=H_{n,p;k}$ does not contain an independent set of size $\bfrac{2k\log d}{(k-1)d}^{\frac{1}{k-1}}n$. 
\end{lemma}
\begin{proof}
Let $u=\bfrac{2k \log d}{(k-1)d}^{\frac{1}{k-1}}n$. The probability that there exists an independent set of size $u$ in $H$ is bounded by 
\begin{align}
\binom{n}{u}(1-p)^{\binom{u}{k}} &\leq \bfrac{en}{u}^u \exp\set{-\frac{d}{\binom{n-1}{k-1}} \cdot \binom{u}{k}}\nonumber\\
&\leq  \bfrac{en}{u}^u \exp\set{-\frac {du}{k} \bfrac{u}{n}^{k-1}\brac{1+O\bfrac{1}{n}}}\nonumber\\
& =\brac{e^{k-1}\frac{(k-1)d}{2k\log d} \cdot \exp\set{-2\log d\brac{1+O\bfrac{1}{n}} }}^{u/(k-1)}\nonumber\\
&= \brac{\frac{e^{k-1}(k-1)}{2kd\log d}\brac{1+O\bfrac{1}{n}}}^{u/(k-1)}\label{d2}\\
&=  o(1).\nonumber
\end{align}
\end{proof}
\begin{notation}
We let 
$$m_0= \frac{n}{\a}\text{ and }n_0=16m_0\log^2d.$$ 
Furthermore, for $t\leq d$ we let 
$$S_t=\set{(s_1,s_2,...,s_t) \in \bigg[ \bfrac{2k \log d}{(k-1)d}^{\frac{1}{k-1}}n \bigg]^t: 
\sum_{j=1}^t  s_i \leq \min\set{t m_0 ,n- n_0 }}.$$
\end{notation}
\begin{lemma}\label{indeph}
If $k\geq 2$ and $d$ is sufficiently large then, w.h.p. {\bf there does not exist} $1\leq t\leq d$ and disjoint sets $V_1,...,V_t\subset V$ such that:
\begin{enumerate}[(i)]
\item $V_1,V_2,\ldots,V_t$ is a maximal independent sequence of length $t$ in $H=H_{n,p;k}$.
\item $\big(|V_1|,|V_2|,...,|V_t|\big) \in S_t$.
\end{enumerate}
\end{lemma}
\begin{proof}
Fix $t\in[d]$, $(s_1,...,s_t)\in S_t$ and let $\bar{s}=\frac{1}{t}  \sum_{i\in [t]} s_i$. Since $(s_1,...,s_t)\in S_t$ we have that $\bar{s} \leq \frac{1}{t} \cdot tm_0=m_0$. There are $\binom{n}{s_1,s_2,...,s_t,n-t\bar{s}}$ ways to pick disjoint sets $V_1,V_2,...,V_t \subseteq V$ of sizes $s_1,...,s_t$ respectively. So $V_1,...,V_t$ satisfy condition (i) of Lemma \ref{indeph} only if for every $i\in [t]$ and every $v \in V \setminus  \underset{ j \in [i] }{ \bigcup} V_{j}$, there exist $u_1,...,u_{k-1}\in V_i$ such that $\{u_1,...,u_{k-1},v\}\in E(H)$.  So, given $V_1,...,V_t$ the probability that we have (i) is at most 
\beq{p1}{
p_1= \prod_{i=1}^t (1-(1-p)^{\binom{s_i}{k-1}})^{n-\sum_{j=1}^i s_j}
\leq \exp\set{-\sum_{i=1}^t \brac{(1-p)^{\binom{s_i}{k-1}} \brac{{n-\sum_{j=1}^i s_j}}}}.
}   
Now let $t'=\max\set{i: \sum_{j\leq i} s_j \leq n-\frac{n}{\log^2 d}}$ and set
$\bar{s}'=\frac{1}{t'}\sum_{i=1}^{t'}s_i$. We consider 2 cases.

{\bf Case 1:} $t'\geq \big(1-\frac{1}{\log d}\big)t$.\\ 
 Now $t\bar{s}\geq t'\bar{s}'$ and so $\bar{s}'-\bar{s}\leq \frac{t-t'}{t} \bar{s}'\leq \frac{\bar{s}'}{\log d}$, which implies that $\bar{s}'\leq \bar{s}\brac{1-\frac{1}{\log d}}^{-1} \leq m_0\brac{1+\frac{2}{\log d}}$. 
Then,
\begin{align*}
&\ \ \ \ \sum_{i=1}^t \brac{(1-p)^{\binom{s_i}{k-1}} \brac{n-\sum_{j=1}^i s_j} }\\
&\geq\sum_{i=1}^{t'} \brac{(1-p)^{\binom{s_i}{k-1}} \brac{n-\sum_{j=1}^i s_j} }\\
&\geq \frac{n}{\log^2d} \sum_{i=1}^{t'} (1-p)^{\binom{s_i}{k-1}}\geq \frac{nt'}{\log^2d} (1-p)^{\binom{\bar{s}'}{k-1}}\geq \frac{nt}{2\log^2d}(1-p)^{\binom{m_0\brac{1+\frac{2}{\log d}}}{k-1}}\\
&\geq \frac{nt}{2\log^2d}\exp\set{-(p+p^2)\binom{\bfrac{(\log d-5(k-1)\log \log d)}{(k-1)d}^{1/(k-1)}\brac{1+\frac{2}{\log d}}n}{k-1}}
\\ &\geq \frac{nt}{2\log^2d} \exp\set{- \frac{\log d-5(k-1)\log \log d}{k-1}\cdot \brac{1+\frac{3(k-1)}{\log d}}} \\
&\geq \frac{nt\log^2d}{d^{1/(k-1)}}.
\end{align*}
Now
$$\binom{n}{s_1,...,s_t,n-t\bar{s}}\leq \binom{n}{\bar{s},...,\bar{s},n-t\bar{s}}\leq \prod_{i=1}^t \binom{n}{\bar{s}} \leq \bfrac{en}{\bar{s}}^{t\bar{s}} 
\leq \bfrac{en}{m_0}^{tm_0}.$$
Thus the probability that for some $t\leq d$ there exist $V_1,...,V_t$ satisfying conditions (i), (ii) of Lemma \ref{indeph} is bounded by
\mults{
\sum_{t=1}^d \sum_{(s_1,...,s_t)\in S_t} \binom{n}{s_1,s_2,..,s_t,n-\sum_{i\in [t]} s_i} p_1
\\  \leq \sum_{t=1}^d \sum_{(s_1,...,s_t)\in S_t} \bfrac{en}{m_0}^{tm_0} 
\exp\set{-\frac{nt\log^2d}{d^{1/(k-1)}}}\leq\sum_{t=1}^d n^t  
\brac{\frac{(e\a)^{(\log d)^{1/(k-1)}}}{d^{\log d}}}^{nt/d^{1/(k-1)}}  =o(1).
}
{\bf Case 2:} $t'< \big(1-\frac{1}{\log d}\big)t$.\\ 
Thus $t-t'\geq \frac{t}{\log d}$. Observe that from Lemma \ref{auxh} we can assume that 
\beq{tg}{
t\geq t'\geq \brac{\brac{1-\frac{1}{\log^2d}} \bigg/ \bfrac{2k \log d}{(k-1)d}^{\frac{1}{k-1}}}-1\geq \frac14\bfrac{d}{\log d}^{\frac{1}{k-1}}.
}
For \eqref{tg} we are using Lemma \ref{auxh} to argue that we need at least this many independent sets to partition a set of size $n\brac{1-\frac{1}{\log^2d}}$. The -1 comes from the fact that the upper bound in the definition of $t'$ may not be tight.
 
Thus,
\begin{equation}\label{ubound}
u=\frac{1}{t-t'}\sum_{i=t'+1}^t s_i\leq \frac{\log d}{t} \cdot  n\brac{\frac{1}{\log^2 d} + \bfrac{2k \log d}{(k-1)d}^{\frac{1}{k-1}}}
\leq 4\bfrac{\log d}{d}^{\frac{1}{k-1}} \cdot \frac{n}{\log d}
\end{equation}
and now with $p_1$ as defined in \eqref{p1} we have
\begin{align*}
p_1&\leq \prod_{i=t'+1}^t (1-(1-p)^{\binom{s_i}{k-1}})^{n-\sum_{j=1}^i s_j} 
\leq \prod_{i=t'+1}^t (1-(1-p)^{\binom{s_i}{k-1}})^{n_0} \\
& \leq \exp\set{-n_0 \sum_{i=t'+1}^t (1-p)^{\binom{s_i}{k-1}}}\\
& \leq \exp\set{-n_0 (t-t') \exp\set{-(p+p^2) \binom{u}{k-1}}} \\
&\leq \exp\set{-n_0 (t-t')\exp\set{-d\bfrac{u}{n}^{k-1}}\brac{1+O\bfrac{1}{n}}}\\
&\leq \exp\set{-n_0 (t-t')\exp\set{-\frac{4^{k}}{\log^{k-2}d}}}\\
&\leq e^{-(t-t')n_0/2}. 
\end{align*}
Thus the probability that for some $t\leq d$ there exist $V_1,...,V_t$ satisfying conditions (i), (ii) of Lemma \ref{indeph} is bounded by
\mults{
P=\sum_{t=1}^d \sum_{(s_1,...,s_t)\in S_t} \prod_{i=1}^{t'} \binom{n-\sum_{j=1}^{i-1}s_j}{s_i} \prod_{i=t'+1}^{t} \binom{n-\sum_{j=1}^{i-1}s_j}{s_i} p_1 
\\  \leq \sum_{t=1}^d \sum_{(s_1,...,s_t)\in S_t} \bfrac{en}{\bar{s}'}^{t'\bar{s}'}  \bfrac{en}{{u}}^{(t-t')u}  
e^{-(t-t')n_0/2}.
}
For sufficiently large $d$, \eqref{ubound} implies $u\leq m_0$ and we also have that $n_0=16m_0\log^{2}d$. Therefore 
$$ \bfrac{en}{{u}}^{(t-t')u}   e^{-(t-t')n_0/4} 
\leq \bfrac{en}{{m_0}}^{(t-t')m_0}   e^{-4(t-t')m_0 \log^2 d}\leq  e^{-3(t-t')m_0 \log^2 d}\leq e^{-3tm_0 \log d}.$$
Furthermore, Lemma \ref{auxh} implies that $\bar{s}'\leq \bfrac{2k \log d}{(k-1)d}^{\frac{1}{k-1}}n\leq 3m_0$. Thus
$$\bfrac{en}{{\bar{s}'}}^{t'\bar{s}'}   e^{-(t-t')n_0/4} 
\leq \bfrac{en}{3m_0}^{3tm_0} e^{-4(t-t')m_0 \log^2 d}\leq \bfrac{en}{3m_0}^{3tm_0} e^{-4tm_0 \log d}\leq e^{-tm_0 \log d}.$$
So,
$$P\leq dn^d e^{-4tm_0\log d}=o(1).$$
\end{proof}
\begin{lemma}\label{densityh}
If $k\geq 2$ and $d$ is sufficiently large then w.h.p. every set $S\subset V$ of size at most $n_0$ spans fewer than $3|S|\log^{3k} d$ edges in $H$. Hence no subset of size at most $n_0$ contains a $3\log^{3k} d$ core.
\end{lemma}
\begin{proof}
Let $L=3\log^{3k}d$. The probability that there exists $S\subset V$ of size at most $n_0$ that spans at least $ t= L|S|$ edges is bounded by
\mults{
\sum_{s=1}^{n_0} \binom{n}{s} \binom{\binom{s}{k}}{t} p^{t} 
\leq \sum_{s=1}^{n_0} \brac{ \bfrac{en}{s}^{\frac{s}{t}} \cdot \frac{e\binom{s}{k}}{t} \cdot \frac{d}{\binom{n-1}{k-1}} }^{t} \leq \sum_{s=1}^{n_0}\brac{\bfrac{en}{s}^{1/L}\frac{eds}{t}  \bfrac{s}{n}^{k-1}}^t
\\=\sum_{s=1}^{n_0}\brac{ \bfrac{s}{n}^{k-1-1/L}\frac{e^{1+1/L}d}{L} }^t=o(1).
}
\end{proof}
\emph{Proof of Theorem \ref{th1}:} Let $\alpha,\b$ be as in \eqref{defab}. We argue next that the properties given by Lemmas \ref{auxh}, \ref{indeph} and \ref{densityh} imply that $H_{n,p;k}$ is $(\a,\b)$-colorable for $d$ sufficiently large. Lemma \ref{randomh} then follows directly from \eqref{cP} and Theorem \ref{semigreedy}.

Consider a sequence of sets $V_1,V_2,\ldots,V_\a$ such that $V_i$ is maximally independent in $[n]\setminus \underset{j<i}{\bigcup}V_j$ for $j\leq \a$ (some of these sets can be empty). It follows from Lemma \ref{indeph} that because $\a m_0=n$, we must have $\sum_{i=1}^\a |V_i|\geq n-n_0$ and then Lemma \ref{densityh} implies that $[n]\setminus \underset{i\leq \a}{\bigcup}V_i$ does not have a $\b$-core. This completes the proof of the first part of the theorem.

When $q\geq \a+2\b+2$ we see that we have $2\b+2$ colors with which to color a hypergraph with no $\b$-core and Theorem 1 of \cite{BP} implies that we need $O(\b n)$ vertex re-colorings to do this. We will encounter at most $\a$ such hypergraphs in our re-coloring. This together with Remark \ref{rem1} shows that there will be $O(\a\b n)$ re-colorings overall and this proves the second part of the theorem. 
\qed

\end{document}